\documentclass[11]{amsart}

\usepackage{enumerate}
\usepackage{amssymb}

\newtheorem{theorem}{Theorem}[section]
\newtheorem{proposition}[theorem]{Proposition}
\newtheorem{lemma}[theorem]{Lemma}
\newtheorem{corollary}[theorem]{Corollary}

\newtheorem{example}{Example}
\newtheorem{definition}[theorem]{Definition}

\theoremstyle{remark}

\newcommand \NN{\mathbb{N}}

\newcommand \eps{\varepsilon}
\newcommand \emp{\emptyset}

\newcommand \vs{\vec{s}}

\newcommand \ie{\emph{i.e.}}
\newcommand \cM{\mathcal{M}}
\newcommand \bB{\mathbb{B}}

\newcommand \cB{\mathcal{B}}

\newcommand \bC{\mathbf{C}}

\newcommand \cV{\mathcal{V}}

\newcommand \cX{\mathcal{X}}

\newcommand \cZ{\mathcal{Z}}

\newcommand \cF{\mathcal{F}}
\newcommand \cS{\mathcal{S}}
\newcommand \restrict{\upharpoonright}

\newcommand{\bM}{\mathbb{M}}

\begin{document}

\title[Cooperative Boolean systems]{Cooperative Boolean systems with generically long attractors I}
\author{Winfried Just and Maciej Malicki}


\address{Winfried Just: Department of Mathematics, Ohio University, Athens, OH 45701, USA \\
Maciej Malicki: Institute of Mathematics of the Polish Academy of Sciences, Sniadeckich 8, 00-956, Warsaw, Poland}
\email{mathjust@gmail.com, mamalicki@gmail.com}



\begin{abstract}
We study the class of cooperative Boolean networks whose only regulatory functions are COPY, binary AND, and binary OR.  We prove that for all sufficiently large~$N$ and~$c < 2$ there exist Boolean networks in this class that have an attractor of length $> c^N$ whose basin of attraction comprises an arbitrarily large fraction of the state space.  The existence of such networks contrasts with results on
various other types of dynamical systems that show nongenericity or absence of non-steady state attractors under the assumption of cooperativity.
\end{abstract}

\maketitle


\section{Introduction}\label{intro}

Understanding the role of feedback is crucial in the study of dynamical systems; see, \emph{e.g.,} \cite{Thieffry, THdA} for relevant surveys.
The absence of negative feedback loops tends to favor steady state attractors. For example, continuous flows without negative feedback loops are known as \emph{monotone} systems. In these systems trajectories converge generically towards an equilibrium under mild regularity hypotheses;  see \emph{e.g.}  \cite{Enciso:Hirsch:Smith:2008,Hirsch:1988,HalSmith}.
Similarly,~\cite{RR, RRT} show that in Boolean networks with asynchronous updating negative feedback loops are necessary for the existence of attracting limit cycles.

Here we study a related question for Boolean networks with synchronous updating, which is the updating scheme originally proposed by Stuart Kauffman in his seminal papers \cite{Kauffman:N1969,Kauffman:JTB1969}.  More specifically, we study the class of \emph{cooperative} such Boolean networks which is defined by the total absence of negative interactions.  Cooperative Boolean networks are the ones whose regulatory functions can be represented by compositions of AND, OR, and COPY functions, without the use of negations.  This is a more restrictive notion than the absence of negative feedback loops, which are defined as containing an odd number of negative interactions.

In Boolean networks with synchronous updating negative feedback loops are not necessary for the existence of non-steady-state attractors, even exponentially long ones.  It is shown in \cite{Arcena, JE1, Sontag:mono:Arxiv2007} that $N$-dimensional cooperative Boolean networks can have attractors of size up to $\binom{N}{\lfloor N/2\rfloor} \sim \frac{2^N}{\sqrt{N}}$. Note, however, that this upper bound implies that no attractor in sufficiently high-dimensional cooperative Boolean networks can comprise a fixed fraction of the state space, whereas without the assumption of cooperativity even the whole state space can form a cyclic attractor
of length~$2^N$.
Moreover, in a randomly chosen $N$-dimensional cooperative Boolean network, with probability approaching one as $N \rightarrow \infty$, a steady-state attractor will be reached from a randomly chosen initial condition in at most two steps~\cite{npcnote}.

This naturally raises the question whether there exist, for given constants~$p,c$ with $0 < p < 1 < c < 2$, cooperative Boolean networks of arbitrarily large dimension~$N$ for which the union of the basins of attraction of limit cycles of length $> 2^c$ comprises a fraction of at least~$p$ of the state space; a property that we call \emph{$p$-$c$-chaos.}  In this paper we answer this question in the affirmative.  The Boolean systems that we construct use only regulatory functions with one or two inputs and are such that with probability $> p$ the trajectories of any two randomly chosen initial conditions will become equal in the long run, a property we call \emph{$p$-coalescence.}  Note that the latter implies that the basin of attraction of a single attractor of length $> 2^c$  comprises a fraction of $>p$ of the state space.

Before we formally state our theorems in Section~\ref{resultsec}, let us make some additional comments on their relation to the literature on the subject.

The dynamics of Boolean networks tends to fall either into the \emph{ordered regime} that is characterized by short attractors, a large proportion of eventually frozen nodes, and low sensitivity to perturbations of initial conditions, or into the \emph{chaotic regime} that is characterized by very long attractors, very few eventually frozen nodes, and high sensitivity to perturbations of initial conditions~\cite{origins}.  Note that $p$-$c$-chaos is a direct formalization of the first hallmark of chaos.  It also implies, for suitable choices of~$p$ and~$c$, a small proportion of eventually frozen nodes (see Proposition~\ref{fluidprop}).  On the other hand, $p$-coalescence implies low sensitivity of long-term behavior to perturbations of
initial conditions.  Thus the dynamics of the Boolean networks we construct here are extremely chaotic in one sense and highly ordered in another sense.

There exists a large body of literature on the \emph{expected} dynamics of so-called \emph{random Boolean networks (RBNs)} with certain restrictions on their connectivity and regulatory functions (see, \emph{e.g.,} the surveys \cite{ACK, Drossel:review, origins}). In these studies one considers a class $\cB$ of Boolean networks with a given (usually uniform) probability distribution and tries to determine the expected values of the length of attractors, the proportion of eventually frozen nodes, or measures of sensitivity to initial conditions.  For example, let $\cB_{N,K}$ be the class of Kauffman's NK-networks, that is, Boolean networks with $N$ variables and the sole restriction that each regulatory function can take at most $K$ inputs.
Then the expected dynamics of~$\cB_{N,K}$ becomes more chaotic as $K$ increases from~1 to~$N$ in terms of all three features described above.
In particular, $p$-$c$-chaos is generic in  $\cB_{N,N}$ when $p < 1$ and $c < \sqrt{2}$. This contrasts sharply with the situation for randomly chosen cooperative Boolean networks without any restrictions on the connectivity for which the above mentioned  results of~\cite{npcnote} show that $p$-$c$-chaos is about as non-generic as possible.  In a similar vein, simulation studies~\cite{Sontag:Laubenbacher} indicate that decreasing the amount of negative feedback in RBNs from~$\cB_{N,K}$ has the effect of decreasing the average lengths of the attractors.

While Boolean networks with fewer inputs per node tend to have on average shorter attractors, it is still possible to construct \emph{quadratic} Boolean networks (that is,
such that each node takes at most two inputs) with attractors whose length is a fixed fraction of the state space of very large
dimensions~\cite{Kaufman:Drossel:EPJ2005,Paul:Drossel:PRE2006}.  Cooperative \emph{bi-quadratic} Boolean networks (that is, such that the number of in- and outputs per node is bounded from above by~2) can still have an attractor of size $>c^N$ for each $c < 2$ and all sufficiently large~$N$~\cite{JE, JENDST}.
However, if in addition to the latter we require that at least a fixed fraction~$\alpha > 0$ of all nodes take \emph{exactly} two inputs, then the length of attractors is bounded from above by $c_\alpha^N$ for some $c_\alpha < 2$.  In particular, $c_1 = 10^{1/4}$ and the bound is sharp~\cite{JE, JENDST}. The constructions in~\cite{JE, JENDST} do not appear to be adaptable to yielding $p$-$c$-chaotic cooperative  Boolean networks for $p+c$ arbitrarily close to~3, which is achieved by the alternative construction described here. The Boolean networks in our construction are also bi-quadratic.  While the latter property is not actually needed for the construction, it adds interest to our results in view of these bounds on attractor length.

\section{Basic definitions}\label{defsec}

Let $[N]$ denote the set $\{1, \ldots , N\}$ and let $2^N$ denote the set of all binary vectors with coordinates in~$[N]$.
An \emph{$N$-dimensional Boolean system $\bB$ (or Boolean network)} is a pair $\bB = (2^N, f)$, where $f: 2^N \rightarrow 2^N$ is a Boolean function.  While $f$ uniquely determines~$\bB$, we will make a careful verbal distinction between
 a \emph{Boolean network,} which is
a pair,
a \emph{Boolean function} $f:2^r \to 2^u$, and a \emph{partial Boolean function} $f: \subseteq 2^r \to 2^u$.

For a Boolean network $\bB = (2^N, f)$,  the updating function~$f$ determines the successor state $\vs(t+1)$ of a Boolean state $\vs(t) \in 2^N$ one time step later according to the formula

\begin{equation}\label{upeqn}
\vs(t+1) = f(\vs(t)).
\end{equation}

Note that $f = (f_1, \ldots , f_N)$, where  the components $f_i: 2^N \rightarrow \{0,1\}$ give us the successor states for individual variables
according to

\begin{equation}\label{upeqni}
s_i(t+1) = f_i(\vs(t)).
\end{equation}

We will sometimes refer to~$f_i$ as the \emph{regulatory function for variable~$i$.}
Since the state space $2^N$ is finite, the trajectory~$\vs(0), \vs(1), \ldots$ of an initial state~$\vs(0)$ must eventually revisit the same state and thus must either reach a \emph{fixed point} $\vs(t)$ ({\ie}  a state with $f(\vs(t)) = \vs(t)$) or a \emph{cyclic attractor} of length $\leq 2^N$.

Now fix a Boolean network~$\bB = (2^N, f)$. For $\vs = (s_1, \ldots, s_N)$ and $i \in [N]$ let $\vs^{\ j*} = (s_1, \ldots , s_{j-1}, 1 - s_j, s_{j+1}, \ldots, s_N)$ be the state obtained from $\vs$ by a one-bit flip at position~$j$.  If for some $\vs$ and~$j$ we have
$f_i(\vs) \neq f_i(\vs^{\ j*})$, then we say that \emph{variable~$j$ is an input of variable~$i$} and \emph{variable~$i$ is an output of variable~$j$.}  We call $\bB$ \emph{quadratic} if each variable takes at least one and at most two inputs
and \emph{bi-quadratic} if each variable takes one or two inputs and has at most two outputs.
 A partial Boolean function~$f$ is \emph{cooperative} if $\vs \leq \vs^{\ +}$ implies $f(\vs) \leq f(\vs^{\ +})$, where $\leq$ denotes the coordinatewise partial order of Boolean vectors.
By considering its conjunctive or disjunctive normal form, it is easy to see that a Boolean function is cooperative iff it can be represented as a composition of the functions COPY, AND, OR and Boolean constants.
Note that $\bB$ is quadratic and cooperative iff each regulatory function is either COPY, binary OR, or binary AND.
We specifically exclude Boolean constants in this definition since our construction works without their help, which leads to a slightly stronger result.

The following fact will be useful in some of our constructions.

\begin{proposition}\label{extendprop}
Let $f: \subseteq 2^r \to 2^u$ be a partial Boolean function.

\smallskip

\noindent
(a) If the domain of~$f$ consists of vectors that are pairwise incomparable in the coordinatewise partial order, then~$f$ is cooperative.

\smallskip

\noindent
(b) If $f$ is cooperative, then~$f$ can be extended to a cooperative total Boolean function $f^*: 2^r \to 2^u$.
\end{proposition}

\begin{proof}
Point (a) follows immediately from the definition of cooperativity.  For the proof of~(b) we can define the $i$-th coordinate $f^*(\vec{s})_i$ of $f^*(\vec{s})$ as the minimum of $f(\vec{s}^{\ +})_i$
over
$\vec{s}^{\ +}$ in the domain of~$f$ with~$\vs \leq \vs^{\ +}$.
\end{proof}

  For  $0 < p < 1 < c < 2$ we say that a Boolean system $\bB$ is
\emph{$c$-chaotic} if it contains an attractor of length  $> c^N$, and we say that $\bB$ is \emph{$p$-$c$-chaotic} if attractors of length~$> c^N$ will be reached with probability $> p$, that is, if a proportion of  $> p$ of all initial states belongs to basins of attraction of attractors of length~$> c^N$.

A variable~$s_i$ is \emph{eventually frozen} along a trajectory if
the value~$s_i(t)$ is fixed for all $\vs(t)$ in the attractor. We say that $\bB$
is \emph{$p$-fluid} if with probability~$>p$ the attractor of a randomly chosen initial state has a proportion of less than $1-p$ eventually frozen variables.

Two initial states $\vec{s}^{\ 0}(0), \vec{s}^{\ 1}(0)$ \emph{coalesce} if there exists a time $t \geq 0$ such that
$\vec{s}^{\ 0}(t) = \vec{s}^{\ 1}(t)$. We call $\bB$ \emph{$p$-coalescent} if any two randomly and independently chosen initial conditions  will coalesce with probability $> p$.

\section{Statement of main result}\label{resultsec}

We will prove the following result that strengthens Theorem~1(i) of~\cite{JE}.

\begin{theorem}\label{mainth}
Given any $0 < p < 1 < c < 2$, for all sufficiently large $N$ there exist  $p$-$c$-chaotic, $p$-coalescent, $N$-dimensional bi-quadratic cooperative Boolean networks.
\end{theorem}

Let us briefly address the second hallmark of chaotic dynamics, few eventually frozen nodes, that was formalized above as $p$-fluidity.  In~\cite{JE}, some additional effort was required to ensure that the $c$-chaotic systems constructed there are also $p$-fluid.  This is not necessary here, since  for sufficiently large~$c$ the property of
$p$-$c$-chaos already implies $p$-fluidity.

\begin{proposition}\label{fluidprop}
Let  $0 <p < 1$.  Then there exists $c_p < 2$ such that for every $c$ with $c_p < c < 2$ every $p$-$c$-chaotic Boolean system is~$p$-fluid.
\end{proposition}

\begin{proof}
Let~$p$ be as in the assumption.  Let $c_p = 2^{p}$, and let $c_p < c < 2$.  Consider an initial condition in the basin of attraction of an orbit of size  $>c^N$.  If a proportion of at least  $1- p$ of nodes are eventually frozen along the trajectory of this initial condition, then the attractor can have length at most $2^{pN}  = c_p^N$, which contradicts the conditions on~$c$.  However, $p$-$c$-chaos implies that a proportion of more than~$p$ initial condition belongs to basins of attractions of periodic orbits of length $> c^N$, and thus implies $p$-fluidity.
\end{proof}

\section{Some Terminology}\label{terminologysec}

Let us begin by introducing some terminology that will be used throughout the proof of Theorem~\ref{mainth}. The set of positive
integers will be denoted by~$\NN$ and
$\log n$ denotes logarithm in base~2.

We generalize the notion $s_i(t)$ to $s_X(t)$, where   $X$ is a subset of the set of all variables. Thus $s_X(t) \in 2^X$ is the vector with coordinates $s_i(t)$ for $i \in X$.

If~$\tau$ is a positive integer, and $f^\tau$ denotes the $\tau$-th iteration of the updating function~$f$ of a given Boolean system, then

\begin{equation}\label{upeqntau}
\vs(t+\tau) = f^\tau(\vs(t)).
\end{equation}

In the Boolean systems we are going to construct, we will often have subsets $X, Y$ of the variables so that the values
of $s_i(t+\tau)$ for $i \in Y$ will depend only on $s_X(t)$.
Thus we will construct Boolean systems that satisfy the following for some  function $g: 2^X  \rightarrow 2^Y$ and a fixed~$\tau$

\begin{equation}\label{fgXY}
s_Y(t+\tau) = f^\tau (\vs(t)) \restrict Y = g(s_X(t)).
\end{equation}
 We  will refer to the property expressed
by~(\ref{fgXY}) by writing that  \emph{$f$ computes $g$ on input $X$ and writes the output to $Y$ after~$\tau$ steps.}

The function~$g$ will sometimes be given a verbal or formal description rather than be expressed by a separate symbol.

If $(r(t): \ t \in \NN)$ is a given sequence of Boolean vectors in $2^Y$, then we will say
that $f$ \emph{writes~$r(t)$  to~$Y$
with probability~$> q$ for $t \geq \tau$} if for a proportion of~$>q$ of all initial conditions $\vs(0) \in 2^N$ we have:

\begin{equation}\label{fwritesg}
\forall t \geq \tau \ \ s_Y(t) = f^t(\vs(0)) \restrict Y = r(t).
\end{equation}

\section{Outline of the proof of Theorem~\ref{mainth}}\label{description}

Equipped with the terminology introduced in the previous section we are now ready to outline the proof of Theorem~\ref{mainth}.  More precisely,
we will actually prove Theorem~\ref{mainth} in this section with the exception of a number of technical lemmas, some formal definitions, and a brief technical description in Section~\ref{BoolLsec} of how the components of our construction fit together.

Let $0 < p < 1 < c < 2$ be as in the assumptions of Theorem~\ref{mainth}.

The proof of our theorem boils down to constructing, for sufficiently large~$N$, a suitable
updating function~$f$ for a Boolean systems $\bB = (2^N, f)$.  We need to assure that~$f$ is cooperative, bi-quadratic, and works as expected, with probability $> p$, for randomly chosen
(pairs of) initial conditions.

The set of Boolean variables $[N]$ will contain  pairwise disjoint sets~$X_i$, indexed by $i \in I$,
where $I = \{0, 1, \ldots , |I|-1\}$. With probability $> p$ the following will hold for all times~$t$ of the form $t = t_0 + k|I|$, where~$t_0$
is some fixed time and $k$ is a nonnegative integer:
With the possible exception of indices~$i$ in a small subset $Q \subset I$, each vector $s_{X_i}(t)$ will code an integer
$v_i(t) \in \{0, \ldots , n-1\}$ for some suitable value of~$n$ that depends on~$N$.  Moreover, again with the possible exception of $i \in Q$, the function~$f$ computes addition of~1 modulo $n-i$  on input $X_i$ and writes the output to $X_i$ after~$|I|$ steps.  Formally, the latter means that for times~$t$ as above

\begin{equation}\label{modaddI0}
\forall \, i \in I \backslash Q \quad v_i(t+ |I|) = v_i(t) + 1 \ mod \ (n - i).
\end{equation}

Let $\cX = \{X_i: \, i \in I\}$ and
let $\cV = \{V_i: \, i \in I\}$ be the functions that decode integers $v_i$ from certain vectors $s_{X_i}$.
We will also let $X = \bigcup_{i \in I} X_i$ and $Y = [N] \backslash X$.
A structure $\cM = (\bB, p, I, Q, n, \cX, Y, \cV, t_0)$ will be called
 an \emph{$M$-system} if it has the properties described above.

In the next section we will define a type of structure that we call an \emph{$MM$-system.}  Essentially, an $MM$-system is an~$M$-system with several
additional parameters and a number of size restrictions on how $N, n, |I|, |Q|, |X|, |Y|$ are related to each other.  A major concern here is that
  we need to assure that~$|Y|$ is small relative to~$|X|$
and that~$|Q|$ is sufficiently small relative to~$|I|$. One of these additional parameters will be~$c$; the other parameters will allow us to control the size restrictions mentioned above.
We will use notations like $\cM(\bB, p, c)$ for $MM$-systems or~$M$-systems with only the immediately relevant
parameters shown. The implicit understanding will be that the remaining parameters are as specified in Definitions~\ref{Mdef} and~\ref{MMdef}
of Section~\ref{MMsec}, where we will prove
that for $MM$-systems $\cM(\bB, p, c)$ of sufficiently large dimensions
the Boolean system $\bB$ is  $p$-$c$-chaotic (Lemma~\ref{keylem}).

Now the proof of Theorem~\ref{mainth} boils down to proving Lemma~\ref{res1lem} of Section~\ref{MMsec}, that is, to constructing an $MM$-system
such that the updating function $f$ is  bi-quadratic and cooperative and such that with probability~$> \sqrt{p}$ the trajectory of a randomly chosen initial state will
reach a specified state at time~$t_0$.   Thus with probability~$>p$ the trajectories of two randomly chosen initial states will coalesce at some time~$t \leq t_0$.

 We will deal with the size restrictions in
the definition of a~$MM$-system by showing that all functions that~$f$ needs to compute are what we call \emph{$L$-functions.}  We will formally define $L$-functions in Section~\ref{Boolsec}; at the moment it suffices to say that
$L$-functions are Boolean functions that can be implemented by cooperative bi-quadratic Boolean input-output systems of relatively low depth without adding too many variables.  The restriction on the depth means that  all necessary computations will require, for sufficiently large~$n$, only $\tau \ll |I|$ steps, so that the output can be written
to the required variables before it will interfere with other computations.

Let us conceptualize the collection of all $X_i$'s as a circular data tape, with $f$ simply copying the vector $s_{X_{i+1}}(t)$ to $s_{X_{i}}(t+1)$ for most~$i$'s, and also copying $s_{X_{0}}(t)$ to $s_{X_{|I|-1}}(t+1)$.

For $t \in \NN$ let $t^* = t \ mod \ |I|$.  We single out some $i_1, i_2 \in I$ with $i_2 = i_1 - \tau_1$ and construct~$f$ in such a way that $f$ computes $v_{i_1} + 1$ modulo $n - t^*$ in~$\tau_1$ steps on input~$X_{i_1}$ and writes the output of this operation to $X_{i_2}$.

Let us for a moment assume this can be done and show how it implies~(\ref{modaddI0}) for all times~$t \geq t_0$ under suitable assumptions on~$t_0$ and
$\vs(t_0)$.  Specifically, let us assume that $t_0 = i_1$  and that all vectors~$s_{X_i}(t_0)$ are coding, with the possible exceptions
of~$i \in Q$.  Then  $f$ computes $v_{i_1}(t_0) + 1$ modulo $n - i_1$  on input~$X_{i_1}$ and writes the output of this operation to $X_{i_2}$ after~$\tau_1$ steps.  In the next step we will have $v_{i_1}(t_0 + 1) = v_{i_1+1}(t_0)$, and thus $f$ computes $v_{i_1 + 1}(t_0) + 1$ modulo $n - (i_1+1)$  on input~$X_{i_1}$ and writes the output of this operation to $X_{i_2}$ after~$\tau_1$ steps.  And so on.  After~$|I|$ time steps, the data tape will have come round circle, and we will have $v_i(t_0+ |I|) = v_i(t_0) + 1 \ mod \ (n - i)$ for all~$i \in I$,
and~(\ref{modaddI0}) follows by induction for all $t \geq t_0$.

The alert reader will have noticed that the variables in $X_{i_2+1} \cup X_{i_2+2} \cup \dots \cup X_{i_1-1}$ are really not needed, since
$s_{X_{i_2}}(t+1)$ will \emph{not} be a copy of $s_{X_{i_2+1}}(t)$. We keep these variables here for the purpose of giving a uniform and more easily readable description of our construction.  A similar remark applies to the variables $X_{i_1+1} \cup X_{i_1+2} \cup \dots \cup X_{i_0-1}$,
where $i_0$ will be defined shortly.

In order to make the idea of~$f$ computing $v_{i_1} + 1$ modulo $n - t^*$ work, we will need a counter for tracking~$t^*$.
The easiest way of implementing the counter is as a set of variables~$R$ so that  $s_R(t)$ will be a code for
$t^* =  t\ mod\ |I|$.

\begin{lemma}\label{f1lem0}
There exists an $L$-function~$F_1$ that computes $v_{i_1} + 1$ modulo $n - i$ in $\tau_1$ steps on input $(X_{i_1}, R)$ and writes its output to $X_{i_2}$ if given input vectors
$(s_{X_{i_1}}, s_{R})$ such that $s_{X_{i_1}}$ codes the integer~$v_{i_1}$ and $s_{R}$ codes the integer~$i$.
\end{lemma}

We will need to assure that at time~$t_0$ the system will reach a state where $s_{R}(t) = b(t)$ for all times~$t \geq t_0$ and the $b(t)$'s are specified Boolean vectors that make the construction work.
This cannot be achieved for all possible initial conditions, only with probability arbitrarily close to one.

In Section~\ref{Boolsec} we will also define a notion of \emph{$L$-sequences.} Essentially, these are periodic sequences of Boolean vectors,   for which there are cooperative bi-quadratic Boolean input-output systems $B$ that return this sequence with probability arbitrarily close to one, for all sufficiently large times.  As for $L$-functions, there are restrictions on the depth and number of internal variables in~$B$ so that we can incorporate~$B$ into an $MM$-system.
Moreover, we require that  the vector of internal variables of~$B$ will eventually assume a sequence of specified states~$h^*(t)$, for all times $t \geq t_0$, also with probability arbitrarily close to one.

\begin{lemma}\label{countlem0}
Any specified sequence $(b(t): \, t \in \NN)$  of values of $s_R$ such that $b(t + |I|) = b(t)$ for all $t \in \NN$ is an $L$-sequence.
\end{lemma}

We also need to assure that the function $F_1$ of Lemma~\ref{f1lem0} receives input vectors~$s_{X_i}$ that do code for integers.
Unfortunately, with probability close to~1, the vectors~$s_{X_i}(0)$ in the initial state will \emph{not} be coding any integers. In fact, our Lemma~\ref{crudepr} shows  that
with probability arbitrarily close to~1 as $N \rightarrow \infty$, each of the vectors~$s_{X_i}(0)$ will be belong to a set of \emph{crude vectors} that we will rigorously define in Section~\ref{codingsec}.

All crude vectors are incomparable with all coding vectors in~$2^{X_i}$. The
latter property turns out to be actually quite useful for our purposes.  It allows us to incorporate a Boolean circuit into our definition of~$f$ so that it takes as inputs $(s_{X_{i_0}}, s_{R})$ and writes an identical copy of the input $s_{X_{i_0}}$ to
$s_{X_{i_1}}$ if the input consists of a pair of coding vectors, but outputs a fixed coding vector if at least one of the input vectors is crude.
The circuit takes $\tau_3$ time steps for its calculations, and we let $i_0 = i_1 + \tau_3$.  Formally:

\begin{lemma}\label{f2lem0}
There exists an $L$-function
$F_3$ on inputs  $(X_{i_0}, R)$ that writes its output to $X_{i_1}$ and returns a specified coding vector $x^*$  on all inputs with the property that $s_{X_{i_0}}$ or  $s_{R}$ is crude and that returns an identical copy of $s_{X_{i_0}}$  on all  pairs of input vectors that code integers.
\end{lemma}

Now let $0 < q_1, q_2 < 1$ be probabilities such that
$q_1 + q_2 - 1 > \sqrt{p}$, that is, such that if events $E, F$ occur with probabilities $> q_1$ and  $> q_2$ respectively, then event $E \cap F$ occurs with probability $> \sqrt{p}$. Let~$N$ be sufficiently large so that with probability $> q_1$ the Boolean input-output system $B_2$ that returns the sequence of $b(t)$'s for all times~$t \geq \tau_2$, where $\tau_2 \ll |I|$, behaves as desired and with probability $> q_2$ all initial vectors $s_{X_i}(0)$ are crude.
We will construct the updating function~$f$ so that it is cooperative, bi-quadratic, computes the functions~$F_1, F_3$ in the sense of~(\ref{fgXY}) with inputs and outputs as specified by Lemmas~\ref{f1lem0} and~\ref{f2lem0}, with appropriate values of~$b(t)$ and, with probability $> q_1$, writes these values~$b(t)$ to
$s_R(t)$ for all times $t \geq \tau_2$.
The formal proof of Lemma~\ref{res1lem} in the remainder of this paper shows that it is possible to construct such~$f: 2^N \rightarrow 2^N$ for all sufficiently large dimensions~$N$ in such a way that~$|I|$ exceeds the combined number of time steps of all necessary computations.

With probability $> \sqrt{p}$, this presents us with the following situation after $t_2 := \max\{\tau_2, \tau_3\}$ time steps: The inputs $s_{X_{i_1}}, s_R$ of~$F_1$ are pairs of codes for  integers as needed
for the computation of $v_{i} + 1$ modulo $n  - i$, since $s_R = b(t_2)$ is the desired value of the counter and $s_{X_{i_1}}$ is the value of $F_3$ for an input with a crude coordinate~$X_{i_0}(0)$.  The analogous property  remains true at  subsequent time steps.  Now let us move the system forward to time $t_1 := t_2 + \tau_1$.  At this time, the vectors $s_{X_{i_2}}, s_{X_{i_2-1}}, \dots , s_{X_{i_2- t_1 + 1}}$ will be copies of output
of $f$ for values of the input vectors on $X_{i_1}, R$ that with high probability are not pairs of codes for integers.  We don't have
much control over these vectors; it is not at all clear whether they are coding or crude.    So we cannot automatically
assume that they will be turned into coding vectors once they will have migrated to position $i_0$.  This is were $Q$ comes in:
We let $Q = \{i_2, i_2-1, \dots , i_2- t_1 +1\}$. One may consider~$Q$ a set of ``possibly corrupted memory locations.''  Our formulation
of~(\ref{modaddI0}) allows us to disregard these memory locations. We will tag these locations by specifying a fixed crude vector~$b(t)$ for all those times~$t$ such that~$F_3$ takes an input vector~$s_{X_{i_0}}$ at time~$t$ which is a copy of some $s_{X_i}(t_1 + m|I|)$ with
$i \in Q$, where $m$ is a nonnegative integer.  The crude input coordinate at these times~$t$ will ensure that~$F_3$ writes the specified coding vector
$x^*$ to $X_{i_1}$ at the corresponding times~$t + \tau_3$.

Now consider $t_0 = t_1 + |I|$.  With probability~$> \sqrt{p}$ the following holds:  All the vectors $s_{X_i}(t_0)$ with $i \notin Q$ will be coding.  For all $t \geq t_0 > t_1$, by Lemma~\ref{countlem0}, the vectors~$s_{R}(t)$ can be required to hold codes $b(t)$ for integers so that the function~$F_1$
will add~1 $mod \ (n - t^*)$ to $v_{i_1}(t)$ and write the output to $X_{i_2}$ at time $t + \tau_1$,
where $t^* = t \ mod \ |I|$, unless $s_{X_{i_1}}(t)$ is a copy of some $s_{X_i}(t_1 + m|I|)$ with $i \in Q$.  The function $F_3$ will simply copy $s_{X_{i_0}}(t)$ to
$s_{X_{i_1}}(t+ \tau_3)$ unless the associated $b(t)$ indicates that $s_{X_{i_0}}(t)$ originates from some $s_{X_i}(t_1 + m|I|)$ with $i \in Q$.  Thus
for $t = t_0$ property~(\ref{modaddI0}) will hold, and by induction, the same will be true for all times~$t$ of the form~$t = t_0 + k|I|$ for some
nonnegative integer~$k$.
Since $q_1 + q_2 - 1 \geq p$, the latter implies that $\bB$ thus constructed can be part of an $MM$-system, and it follows from Lemma~\ref{keylem} that~$\bB$ is $p$-$c$-chaotic.

A slight technical difficulty arises from the need for $F_1$ and $F_3$ to receive the signal of a corrupted memory location given by a crude~$b(t)$ at different times; we will show in Section~\ref{BoolLsec} how to implement the necessary phase shift.

For the proof of $p$-coalescence, we will need in addition that with probability~$> \sqrt{p}$ at  time
$t_0$ all variables in the system will assume values  specified by a fixed  vector~$\vs^{\, +}$.  The argument for this works for randomly chosen initial states with probability $> q_1 + q_2 - 1 > \sqrt{p}$ and
goes as follows.  Our set of variables
is $X \cup Y$, where $Y$ comprises the internal variables of the Boolean input-output systems $B_1, B_3$ that implement~$F_1, F_3$, and $B_2$ that returns the sequence of $b(t)$'s, plus some dummy variables.   We may choose $\tau_2$ such that already starting from time $t_2 < t_1$, with probability $> q_1$ the vector of all internal variables of~$B_2$ will assume a fixed sequence of values~$h^*(t)$ for all times $t \geq t_0$ by our definition of an $L$-sequence.  With probability~$> q_2$, all initial values $s_{X_i}$ will be crude; as long as $F_3$ receives an unaltered copy of these as its first input, it will turn it into the specified coding vector~$x^*$.  Similarly, the vectors $s_{X_i}(t_1)$ for $i \in Q$
will be copied to inputs of $F_3$ and converted by $F_3$ into the value $x^*$ after $\tau_3$ time steps.  All these operations will be completed before time~$t_0$.  Thus at time~$t_0$, all vectors $s_{X_i}$ will have values specified by the fixed sequence of $b(t)$'s and the output of~$F_1$ on input
vectors $(x^*, b(t))$.  Moreover, the states of the internal variables of the Boolean circuit~$B_1$ at time~$t_0$ will be determined by the sequence of the inputs between times $t_0 - \tau_1$ and $t_0 - 1$, which is fixed.  A similar observation holds for the the states of the internal variables of the Boolean circuit~$B_3$ at time~$t_0$. The value of any other variable at time~$t$ will simply be a copy of the value of some output variable of~$B_2$ or of some variable in~$X$ at time~$t - \tau$, where $\tau < \tau_3$.  So we will get coalescence on those variables as well, and for a proportion of~$> \sqrt{p}$ of initial states, $\vs(t_0)$ will be the same state~$\vs^{\, +}$.

\section{$MM$-systems}\label{MMsec}

Let us begin by formally stating the definitions of an $M$-system and an $MM$-system that were  described somewhat informally in the previous section.

\begin{definition}\label{Mdef}
An $N$-dimensional \emph{$M$-system} will be a mathematical object
$\cM = (\bB, p, I, Q, n, \cX, Y, \cV, t_0)$ such that

\begin{itemize}
\item $\bB$ is an $N$-dimensional Boolean system.
\item $p$ is a probability such that $0 < p < 1$.
\item $I = \{0, 1, \ldots , |I|-1\}$ is a set of consecutive nonnegative integers.
\item $Q \subset I$.
\item $n$ is a positive integer such that $\log(n)$ is an integer.
\item $\cX = \{X_i: \ i \in I\}$ is a family of pairwise disjoint subsets of the set $[N]$ of Boolean variables of $\bB$.
\item $Y = [N] \backslash \bigcup \cX$.
\item $\cV = \{V_i: \ i \in I\}$ is a family of partial functions such that $V_i$ maps a subset of $2^{X_i}$ onto $\{0, \ldots , n-1\}$.
\item $t_0$ is a positive integer.
\item If $s_{X_i}$ is in the domain of $V_i$, then $V_i(x_i)$ will be denoted by $v_i$.
\item For a proportion of $> p$ of all initial conditions the following will hold for all $t$ of the form $t =  t_0 + k|I|$, where $k$ is a nonnegative integer and $t_0$ is a fixed time, and all $i \in I \backslash Q$:

\smallskip

    $s_{X_i}$ is in the domain of~$V_i$ and

\begin{equation}\label{modaddI}
v_i(t+ |I|) = v_i(t) + 1 \ mod (n - i + 1).
\end{equation}

\end{itemize}
\end{definition}

As mentioned in the previous section, an $MM$-system is  an $M$-system with certain additional parameters and size restrictions
that will imply~$p$-$c$-chaos in~$\bB$.

\begin{definition}\label{MMdef}
An~$N$-dimensional \emph{$MM$-system} is a mathematical object
$\cM = (\cM^-, c, \eps, \delta, \beta, \nu)$ with the following properties:

\begin{itemize}
\item $\cM^- = (\bB, p, I, Q, n, \cX, Y, \cV, t_0)$ is an~$N$-dimensional $M$-system.
\item $I = \{0, 1, \ldots , \beta \log(n)-1\}$.
\item $c, \eps, \delta, \beta, \nu$ are constants such that $1 < 2 < c$; $\beta, (1 + \eps)\log(n)$ are positive integers, $\eps, \delta > 0$, $\beta > \nu \geq 0$, and
\begin{equation}\label{logest}
\log (c)(1 + \eps + \delta) < 1.
\end{equation}

\begin{equation}\label{Xisize}
\forall i \in I \ |X_i| = (1 + \eps)\log(n).
\end{equation}

\begin{equation}\label{Yest}
|Y| \leq \left(\beta\delta - \frac{\nu}{\log c}\right)\log^2(n).
\end{equation}

\begin{equation}\label{Qest}
|Q| \leq \nu \log(n).
\end{equation}

\end{itemize}
\end{definition}

\bigskip

Assume that $\cM$ is an $N$-dimensional $MM$-system.  By~(\ref{Xisize}) and the choice of~$I$ we have

\begin{equation}\label{sizeX}
|X| = \beta(1 + \eps) \log^2 (n).
\end{equation}

Moreover, by~(\ref{sizeX}) and~(\ref{Yest}) we will have

\begin{equation}\label{sizeN}
 \beta(1 + \eps) \log^2 (n) \leq N \leq \left(\beta(1 + \eps + \delta) - \frac{\nu}{\log c}\right) \log^2(n).
\end{equation}

Note also that by our choice of~$I$, equation~(\ref{modaddI}) becomes

\begin{equation}\label{modadd}
v_i(t+ \beta \log(n)) = v_i(t) + 1 \ mod (n - i + 1).
\end{equation}

It will often be more convenient to use~(\ref{modaddI}) written in the form~(\ref{modadd}).

\begin{lemma}\label{keylem}
Let $p, c$ be  constants such that $0 <p < 1 < c < 2$.  Then  for all sufficiently large~$N$ and every
$N$-dimensional $MM$-system $\cM(\bB, p, c)$, the Boolean system~$\bB$ is $p$-$c$-chaotic.
\end{lemma}

\begin{proof}
Let $\cM(\bB, p, c)$ be an $N$-dimensional $MM$-system with $N$ (and hence~$n$) sufficiently large.

Let $\vs(0)$ be a randomly chosen initial condition and let $A$ be the attractor that the trajectory of~$\vs(0)$ eventually reaches.
With probability~$> p$, condition~(\ref{modadd}) will hold for all~$t$ of the form~$t = t_0 + k|I|$ such that $\vs(t) \in A$. If this is the case, then we have $\vs(t_1\beta\log(n)) \neq \vs(t_2\beta\log(n))$ whenever $\vs(t_1\beta\log(n)) \in A$ and  $t_1 < t_2$ are such that
$|t_2 - t_1| < n^{-|Q|}LCM(n, n-1, \dots , n - \beta log(n) + 1)$.  Thus

\begin{equation}\label{attest}
|A| \geq \frac{LCM(n, n-1, \dots , n - \beta log(n) + 1)}{n^{\nu\log(n)}}.
\end{equation}

It is known that
\[ {\rm LCM}(n(n-1)\ldots (n-k)) \geq \frac{(n-k)^k}{k!} \]
for any $n>k \in \NN$ (see Theorem~4 of \cite{Fahri}). By Stirling's formula,

\[ \frac{(n-k)^k}{k!} \approx \frac{(n-k)^ke^k\sqrt{2\pi k}}{k^k} \geq \frac{(n-k)^k}{k^k}\frac{e^k\sqrt{2\pi k}}{k^k}
\geq \frac{n^k}{k^{2k}} \geq 2^{k\log(n)-2k\log(k)}, \]

so, for $k=\beta \log(n)$, we have
\begin{equation}\label{LCMest}
LCM(n(n-1) \ldots (n-\beta\log(n))) \geq 2^{\beta \log^2(n)-2\beta \log(n)\log(\beta\log(n))}.
\end{equation}

From~(\ref{attest}) and~(\ref{LCMest}) we get

\begin{equation}\label{Aest}
|A| \geq 2^{(\beta - \nu) \log^2(n)-2\beta \log(n)\log(\beta\log(n))}.
\end{equation}

On the other hand, by~(\ref{sizeN}) we have

\begin{equation}\label{cnest}
c^N <  2^{(\log (c)(1 + \eps + \delta)\beta  - \nu) \log^2(n)}.
\end{equation}

For fixed $\beta, \nu$, the term
$2\beta\log(n)\log(\beta\log(n))$ becomes negligible as long as $n$ is sufficiently large. Thus, by~(\ref{logest})
the right-hand side of~(\ref{Aest}) will exceed the the right-hand side of~(\ref{cnest}), and Lemma~\ref{keylem} follows.
\end{proof}

The following lemma implies  Theorem~\ref{mainth}.  In view of Lemma~\ref{keylem}, its first part implies $p$-$c$-chaos, and its last sentence implies $p$-coalescence.

\begin{lemma}\label{res1lem}
Let $p, c$ be constants such that $0 < p < 1 < c < 2$.  Then there exists  a positive integer $N_{p,c}$ such that for all $N > N_{p,c}$ there exists an $N$-dimensional $MM$-system
$\cM(\bB, p, c, t_0)$ such that $\bB$ is bi-quadratic and cooperative.
Moreover, this system can be constructed in such a way that there exists a state $\vs^{\, +}$ so that for a randomly chosen initial condition $\vs(0)$ we will have $\vs(t_0) = \vs^{\, +}$ with probability $> \sqrt{p}$.
\end{lemma}

\section{Boolean input-output systems}\label{Boolsec}

A \emph{Boolean input-output system} is a hierarchical arrangement of binary variables,
with the input variables constituting the lowest level, the output variables the highest level and each  variable (except for the ones at the input level)
taking input from one or several variables at some level(s) other than the output level and updating their current state according to an AND, OR, NOT gate, or simply
copying its single input.

A \emph{Boolean circuit} is a Boolean input-output system in which all the variables (except for the ones at the input level)
take input from one or several variables at the next lower level.

The \emph{depth} $d(B)$ of a Boolean input-output system~$B$ is the number of its levels, not counting the lowest (input) level.
An \emph{internal variable} of a Boolean input-output system is a variable that is neither an input variable nor an output variable.
We will  use the notation $B(D, H, R)$ to indicate the sets of variables of a Boolean input-output system~$B$.  In this notation,
the sets $D, H, R$ are pairwise disjoint, the domain $D$ contains the input variables, the range $R$  denotes the set of output variables, and $H$ (for \emph{hidden}) denotes the
set of internal variables.
Notice that the depth of~$B(D, H, R)$ is~1 iff $H = \emp$.

We will call a Boolean input-output system~$B = B(D, H, R)$ \emph{cooperative} if it does not use NOT gates and  \emph{quadratic} if it uses only binary or unary gates.
 A quadratic system is \emph{bi-quadratic} if each variable in $D \cup H$ serves as input for at most two variables
in the system.
We say that $B$ is \emph{monic} if every gate takes only one input.  Note that this definition allows $B$ to be simultaneously monic and bi-quadratic.

Let $B(D, H, R)$ be a Boolean circuit of depth~$d$ and let $g: \subseteq 2^{D} \rightarrow 2^R$ be a (partial) Boolean function.
We say that  $B$
\emph{calculates~$g$ (in $d$ steps)} if for every input $\vs$ in the domain of~$g$
after $d$ updating steps the value of the output vector of~$B$ will be $g(\vs)$.

\begin{example}\label{copyex0}
Let $D, R$ be such that $|D| = |R|$ and let  $\chi : D \rightarrow R$ be a bijection.  Define $id: 2^{D} \rightarrow 2^R$ as the function that extends $\chi$ to binary vectors in the obvious way.
Then there exists a cooperative bi-quadratic Boolean circuit $B_{id} = B_{id}(D, \emp,  R)$ of depth~1 that calculates the function~$id$.
 \end{example}

We can construct $B_{id}$ in such a way that for $i \in D$ the value of $s_i(t)$ will simply be copied to $s_{\chi(i)}(t+1)$ by a monic regulatory function.   This makes $B_{id}$ cooperative, bi-quadratic, and monic.

Note that since in a Boolean circuit variables at each level take input only from the variables at the next lower level, the output of a Boolean
circuit of depth~$d$ (after~$d$ steps) will not be influenced by the initial state of its internal variables.  In contrast,  the output of a Boolean input-output system after $d$ time steps  may also depend on the initial states of its internal variables.

Let $B(D, H, R)$ be a Boolean input-output system of depth~$d$ and let $(s_R(t): \ t \in \NN)$ be a sequence of  Boolean vectors in~$2^R$.
We say that  $B$
\emph{returns $s_R(t)$ with probability $> q$ for all $t \geq d(B)$} if for a proportion of $> q$ of all initial states of the internal variables of~$B$ the output sequence will be as specified for all times $t \geq  d(B)$, for any sequence $(s_D(t): \, t \in \NN)$ of values of the input variables.  This makes the input variables redundant, but for some constructions in the follow-up paper~\cite{JMII} we will need to include nonempty sets~$D$ of input variables.  For this reason we give the definition here in its full generality rather than restricting it to the case~$D = \emptyset$, which is the only one needed here.

Boolean circuits $B_1(D_1, H_1, R_1)$ and $B_2(D_2, H_2, R_2)$ can be concatenated in the obvious way as long as $R_2 = D_1$.  The concatenation $B = B_1\circ B_2$ will have input variables
$D_2$, output variables $R_1$, internal variables $H = H_2 \cup R_2 \cup H_1$, and its depth will satisfy $d(B) = d(B_1) + d(B_2)$.  Moreover,
if $B_1, B_2$ are Boolean circuits that calculate $g_1, g_2$ respectively, then
$B$ will calculate $g_1 \circ g_2$.   One can also convert Boolean input-output systems into parts of an updating function~$f$ of a Boolean network by re-interpreting the logical gates as updating functions for the relevant variables.  The computations performed by the Boolean input-output systems then guarantee the corresponding computations of the updating function~$f$ in the sense of~(\ref{fgXY}).

Now we are ready to define $L$-functions and  $L$-sequences.  Since we are interested in constructing Boolean systems of all sufficiently large dimensions, strictly speaking, these objects are really families of functions or sequences that contain one representative for
each desired dimension~$N$.  In order to keep the terminology in our already rather technical proof reasonably manageable, we suppress reference to the whole family whenever this seems to make our arguments more transparent.  The size of the domains and ranges of individual $L$-functions ($L$-sequences) will be
controlled by the parameter~$n$ in the definition of an $MM$-system. In Section~\ref{BoolLsec}   we will define a set $\bM$ of \emph{suitable}~$n$, which is, a set of all $n \in \NN$ that could appear as parameter in at least one~$MM$-system of interest.  Our families of $L$-functions and $L$-sequences will be indexed by this set.

\begin{definition}\label{Ldef}
Let $\cF = (F^n: \ n \in \bM)$  be a family of functions with $F^n: D(n) \rightarrow R(n)$.
We say that $\cF$ consists of \emph{$L$-functions} if there exist $\gamma >0$ such that for every $n \in \bM$
there exists a cooperative bi-quadratic Boolean circuit with $\leq \gamma (log(n))^2$  variables of depth $\leq \gamma \log(n)$ that calculates
$F^n$.
\end{definition}

\begin{definition}\label{Ldefstat}
Let $\cS = (S(n): \ n \in \bM)$  be a family of sequences with $S(n) = (s^n_R(t): \ t \in \NN)$.
We say that $\cS$ consists of \emph{$L$-sequences} if for every given probability $q < 1$ there exists $\gamma >0$ such that for every $n \in \bM$
there exists a cooperative bi-quadratic Boolean input-output system $B = B(\emptyset, H, R)$ with $\leq \gamma (log(n))^2$  variables of depth $d \leq \gamma \log(n)$ that \emph{returns $s^n_R(t)$ with probability $> q$ for all $t \geq d(B)$}, and a fixed sequence of vectors $h^*(t) \in 2^H$  such that with probability~$> q$ the equality $s_H(t) = h^*(t)$ will hold for all $t \geq 2d(B)$.
\end{definition}

Note that certain size restrictions on~$D(n)$ and~$R(n)$ will be necessary to prove that a given sequence~$\cF$ is an~$L$-function.  Similarly, certain size restrictions on~$s^n_R(t)$ and the periods~$T_n$ will be needed to prove that a given~$\cS$ that consists of sequences~$s^n_R(t)$ with periods~$T_n$ is an $L$-sequence.  These size restrictions have not been spelled out explicitly in Section~\ref{description}; they will be made explicit in Section~\ref{BoolLsec}, and will also be spelled out in Section~\ref{techproofsec}, where we prove Lemmas~\ref{f1lem0},~\ref{countlem0}, and~\ref{f2lem0}.  Let us just mention here that for a periodic sequence to be an $L$-sequence it is actually sufficient that
$|s^n_R|T_n$ grows sufficiently slowly relative to~$n$.

\section{Coding vectors and crude vectors}\label{codingsec}

For an even positive integer~$k$ let $C_k$ be the set of all Boolean vectors from $2^k$ such that exactly half of their coordinates are $1$'s (so the other half are $0$'s).  Then $(C_k)^\ell$ is a set of
 pairwise incomparable vectors for every positive integer~$\ell$. For every $i \in I$ we will
 choose  sets $\bC_i \subset 2^{X_i}$ of coding vectors with $\bC_i \subset (C_k)^\ell$ for some suitable~$\ell$.
 Of course, the sets of variables $X_i$ are pairwise disjoint, so we cannot literally
 make each $\bC_i$ a subset of $(C_k)^\ell$; formally we will need disjoint copies of  $(C_k)^\ell$.  However, adding an extra parameter (as in:
  $(C_k(i))^\ell$) appears to introduce only clutter and we will use the slightly informal notation for the sake of transparency.  Similarly, we will require that if $b_i \in 2^{R_i}$ codes for an integer, then $b_i \in (C_k)^{m}$
  for some suitable~$m$.  This assures that all our codes for integers will be pairwise incomparable and allows us
  to construct cooperative Boolean functions since every partial Boolean function that is defined on a set of pairwise incomparable Boolean
vectors can be extended to a cooperative Boolean function (see Proposition~\ref{extendprop}).

On the other hand, we need to be able to define values $v_i \in \{0, \ldots , n-1\}$ for $x_i \in \bC_i$ so that we can achieve~(\ref{modaddI}).  Thus we want $\bC_i$ to be sufficiently large so that there exists a bijection $V_i: \bC_i \rightarrow \{0, \ldots , n-1\}$.  We will show next that
this is always possible if we judiciously choose~$k$ and the parameter $\eps$ of our $MM$-system.  The vectors $x_i$ in the sets $\bC_i$ will henceforth be called \emph{coding vectors.}

\begin{definition}\label{frienddef}
Let $c$ be a constant with $1 < c < 2$ and let $\eps > 0$.
We say that the pair $(k, \eps)$ is \emph{$c$-friendly} if

\begin{equation}\label{epsrat}
\eps \mbox{ is rational and } \frac{k}{1+\eps} \mbox{ is an integer,}
\end{equation}

\begin{equation}\label{logceps1}
\log(c) (1 + \eps) <  1, \ \mbox{ and}
\end{equation}

\begin{equation}\label{friend1eq}
|C_k| \geq 2^{k/(1+\eps)}.
\end{equation}
\end{definition}

\begin{lemma}\label{friendlem1}
Suppose $1 < c < 2$.  Then there exist a rational $\eps = \eps(c) > 0$ and a positive even integer $k = k(c)$ such that
 the pair $(k, \eps)$ is $c$-friendly.
\end{lemma}

\begin{proof}
Fix any rational $\eps(c) > 0$ such that $\log(c) (1 + \eps(c)) <  1$.  Note that
$|C_k| = \binom{k}{k/2} \geq \frac{2^k}{k}$.  Thus for sufficiently large even~$k$ we will have $|C_k| \geq 2^{k/(1+\eps)}$.
\end{proof}

\section{The proof of Lemma~\ref{res1lem}}\label{BoolLsec}

Having defined all its ingredients, let us formally describe the construction of $MM$-systems  for the proof of Lemma~\ref{res1lem}.
Given~$0 < p < 1 < c < 2$, pick $\eps$ and $k$ such that the pair~$(k, \eps)$ is  $c$-friendly.
An integer~$n$ will be called \emph{suitable} if there exists an integer~$\ell$ such that

\begin{equation}\label{nellchoice}
(1 + \eps)\log(n) = k\ell
\end{equation}

The set~$\bM$ in Definitions~\ref{Ldef} and~\ref{Ldefstat} will be the set of all suitable~$n$.

Pick~$\delta$ with $0 < \delta_0 < \delta < 1$ such that~(\ref{logest}) holds. The latter is possible by~(\ref{logceps1}). We will apply Lemma~\ref{countlem0} to a periodic sequence with period~$T$ of vectors $s_R \in 2^R$ such that
for all sufficiently large~$N$:

\begin{equation}\label{Rsizeprelim}
|R|T \leq \delta_0 \log(n).
\end{equation}

 Let $\gamma_1,
\gamma_3 > 0$ be constants that witness that the functions~$F_1,  F_3$ of Lemmas~\ref{f1lem0} and~\ref{f2lem0} are $L$-functions, and let $\gamma_2$ be a constant that witnesses that the sequence of values of the counter is an $L$-sequence whenever the size restriction~(\ref{Rsizeprelim}) holds.
Let $\gamma = \gamma_1  + \gamma_2 + \gamma_3 + 1$, and let $\nu = \gamma_1 +  \gamma_2 + 2\gamma_3$.

 Choose a sufficiently large positive
integer $\beta$ such that

\begin{equation}\label{Yest1}
\gamma \leq \left(\beta\delta_0 - \frac{\nu}{\log c}\right).
\end{equation}

For sufficiently large~$N$, choose a suitable~$n$ such that

\begin{equation}\label{sizeN0}
 \beta(1 + \eps + \delta_0) \log^2 (n) \leq N \leq \left(\beta(1 + \eps + \delta) - \frac{\nu}{\log c}\right) \log^2(n).
\end{equation}

This determines the set $I = \{0, 1, \ldots , \beta \log(n)-1\}$ and allows us to choose sets of variables
$X_i$ for $i \in I$ with $X = \bigcup_{i \in I} X_i$ such that~(\ref{Xisize}) and~(\ref{sizeX}) hold.

  For each $i \in I$ we will identify a set of pairwise incomparable coding vectors $\bC_i$ of~$2^{X_i}$ that correspond to a subset of a copy of $(C_k)^\ell$ and a bijection $V_i : C_i \rightarrow \{0, \ldots , n-1\}$ that computes the values $v_i(x_i)$ that are coded by $x_i \in \bC_i$.  The existence of $\bC_i$ and $V_i$ follows from~(\ref{friend1eq}) and the choice of~$n, \ell$.
 Now we can choose $m$ as the smallest integer with~$|(C_k)^m| \geq |I|$ and choose a set $\bC \subseteq (C_k)^m$ of codes~$b_i$ for the integers~$i \in I$. Finally, we let $R$ be a set of size $km$; we will
 treat $b_i$ as an element of $2^{R}$.
 It follows from our choice of~$m$ that for some constant $\gamma_R$ we have

\begin{equation}\label{Rsize}
|R| \leq \gamma_R \log(\log(n)).
\end{equation}

The period of the counter will be $T = |I| = \beta \log(n)$. Thus $T$ and~$|R|$ depend on~$\beta$, but regardless of the choice of $\beta$ the estimate~(\ref{Rsizeprelim}) will hold for sufficiently large~$N$.

Let us fix~$i$ and write $X_i \cup R$ as a disjoint union of consecutive intervals $Z^j_i$ of length~$k$ each, where $j < \ell + m$ and let $z_i^j$ denote the respective truncations of $z_i \in 2^{X_i \cup R}$ to these intervals.
 Let us say that $z_i$ is \emph{crude} if there exist $j, J < \ell + m$ such that $z_i^j = \vec{0}$ and $z_i^J = \vec{1}$.  Similarly $x_i \in 2^{X_i}$ is crude if we can find $j, J < \ell$ with this property.  Note that all crude
vectors are incomparable with all coding vectors.

Consider a randomly chosen initial state, and let $E$ be the event that all~$x_i(0)$ are crude.

\begin{lemma}\label{crudepr}
Let $k, \eps, \beta$ be fixed. Then $P(E) \rightarrow 1$ as $n \rightarrow \infty$.
\end{lemma}

\begin{proof}
For each $i$, the probability that $x_i(0)$ is not crude is $\leq 2 (1 - 2^{-k})^\ell$.
Thus the probability of the complement of~$E$ is
$\leq 2\beta \log(n) (1 - 2^{-k})^\ell$.  Since $\log(n) = \frac{k}{1 + \eps} \ell$  and $k, \eps, \beta$ are fixed,
the result follows.
\end{proof}

Let $q_1, q_2$ be probabilities such that
 if events $E, F$ occur with probabilities $> q_1$ and  $> q_2$ respectively, then event $E \cap F$ occurs with probability $> \sqrt{p}$.
In particular, we  will assume that~$N$ is sufficiently large so that
$P(E) > q_1$.

Let $i_0 = |I|-1$ and choose two disjoint sets of variables~$R, R_c$ with $|R|= |R_c|$ as in~(\ref{Rsize}). Choose a Boolean circuit $B_3 = B_3(X_{i_0} \cup R, H_3, X_{i_1})$  that witnesses
that the function~$F_3$ of Lemma~\ref{f2lem0} is an $L$-function, where $i_1 = i_0 - d(B_3)$.  Then  choose a Boolean circuit
$B_1 = B_1(X_{i_1} \cup R_c, H_1, X_{i_2})$  that witnesses
that the function~$F_1$ of Lemma~\ref{f1lem0} is an $L$-function, where $i_2 = i_1 - d(B_1)$.
By Definition~\ref{Ldef}  our choice of the input sets of the Boolean input-output systems implies in view of~(\ref{Rsize}) that we will have $\tau_1 = d(B_1) \leq \gamma_1(\log(n) + \gamma_R \log(\log(n)))$ and $\tau_3 = d(B_3) \leq \gamma_3(\log(n) + \gamma_R \log(\log(n)))$.
For sufficiently large~$n$ the terms $\gamma_R \log(\log(n))$ become negligible, and our choices of $\nu, \beta$ and $I$ imply that $i_2 > 0$.

Choose a Boolean input-output system
 $B_2 = B_2(\emptyset, H_2, R)$ that witnesses
that the sequence~$b(t) \in 2^R$ of desired values of the counter is an $L$-sequence for the function that returns $b(t) \in 2^R$ as needed for the argument presented in Section~\ref{description}.  Make sure that the sets $H_1, H_2, H_3, R, R_c, X$ are pairwise disjoint. These choices
determine the regulatory functions for all variables in $B_1 \cup B_2 \cup B_3$, except for the variables in $X_{i_0} \cup R_c$.

For all $i \in I \backslash \{i_1, i_2\}$ define the regulatory functions for the variables in $X_i$ so that they implement the circuit
$B_{id}(X_{i+1}, \emptyset, X_i)$ of Example~\ref{copyex0}, where $|I|-1 + 1$ is treated as zero.

Notice that for the proper working of $B_3$ it only matters whether the value $s_R$ at a given time~$t$ is coding or crude, where the latter
property signals that the input variables in $X_{i_0}$ originate from a possibly corrupted memory location $X_i$ with $i \in Q$. As already mentioned in Section~\ref{description}, the circuit~$B_1$ needs to receive the same signal $\tau_3 = d(B_3)$ time steps later.  This is where $R_c$ comes in:
We make $R_c$ the set of output variables of a Boolean circuit $B_4 = B_4(R, H_4, R_c)$ of depth $\tau_3$ that is the concatenation of~$\tau_3$ copies of~$B_{id}$ and simply copies the values of the input variables in~$R$ to the output variables in~$R_c$ in~$\tau_3$ steps.

This describes the part of our construction that is needed for $p$-$c$-chaos in the resulting system~$\bB$.  Let $\tau_2 = d(B_2)$.
Note that as long as~$N$ is chosen sufficiently large, we can assume that with  probability $> q_2$ for all times~$t \geq \tau_2$ the values $s_R(t)$ will be the required $b(t)$'s.  Thus with probability $> q_2$, at all times~$t > \tau_2 + \tau_3$ the circuit~$B_1$ will receive
the required input values~$b(t)$ on its input variables in~$R_c$.  Moreover, at all times~$t > \tau_3$ the inputs $s_{X_{i_1}}$ of $B_1$ will be coding vectors.  Thus with probability~$> \sqrt{p}$ the circuit~$B_1$ will write the desired output to~$X_{i_2}$ for all times~$t > \tau_1 + \tau_2 + \tau_3$.  It follows that the set~$Q$ of possibly corrupted memory locations has cardinality~$\leq \tau_1 + \tau_2 + \tau_3$.  Since
$\tau_i \leq \gamma_i \log(n)$ by the definitions of $L$-functions and $L$-sequences, our choice of~$\nu$ implies~(\ref{Qest}), and the choice
of~$\beta$ implies that the block of corrupted memory locations does not extend all around the ``data tape,'' that is, implies~$0 \notin Q$.
Now the argument presented at the end of Section~\ref{description} shows that condition~(\ref{modaddI}) will hold with probability~$> \sqrt{p} > p$ and that we get $p$-coalescence on the variables for which have already defined regulatory functions.

The part of the system we have constructed so far consists of variables in the sets $X = \bigcup_{i \in I} X_i$ and
$Y^- := R_c \cup R \cup H_1 \cup H_2 \cup H_3 \cup H_4$ and regulatory functions for them.
Also, since all Boolean input-output systems that we have used so far are cooperative and bi-quadratic, the part of the system that we have constructed so far inherits these properties.

By Definitions~\ref{Ldef} and~\ref{Ldefstat} and our choice of $\gamma$ we have

\begin{equation}\label{Y-est}
|Y^-| \leq \gamma \log^2(n)
\end{equation}
as long as~$n$ is sufficiently large so that $\gamma_R \log(\log(n)) \leq \log(n)$.  We may need to add to~$Y^-$ a set of dummy variables so that
the resulting set $Y$ satisfies $|X| + |Y| = N$.  By~(\ref{sizeN0}) and~(\ref{sizeX}), the resulting~$Y$ will satisfy~(\ref{Yest}).
 We simply define the regulatory function for each dummy variable to be the function that copies a value
from~$X \backslash (X_{i_0} \cup X_{i_1})$.  This retains cooperativity of the whole system.
Since so far each variable in $X$ serves as input to at most one other variable and since
$|Y| < |X| - 2\log(n)$, we have enough different input variables at our disposal to assure that the resulting~$\bB$ will be bi-quadratic.

The resulting structure satisfies all conditions of an $MM$-system, which proves Lemma~\ref{res1lem}.

\section{The proofs of Lemmas~\ref{f1lem0},~\ref{countlem0}, and~\ref{f2lem0}.}\label{techproofsec}

\subsection{Preliminary results}

\begin{proposition}\label{duplprop}
\label{le1}
Let $\Delta: \left\{0,1\right\} \rightarrow 2^Y$ be the function that produces $|Y|$ identical copies of a variable, defined by $(\Delta(x))_y=x$ for $y \in Y$.  Then there exists a bi-quadratic cooperative Boolean circuit~$B_\Delta^{|Y|}$ with $\leq 2\left|Y\right|$ internal variables that calculates~$\Delta$ in~$\leq \log(\left|Y\right|)$ steps.
\end{proposition}

\begin{proof}
We consider only bi-quadratic systems, so at one step the content of a variable can be copied to two variables only. It means that we need $\left|Y\right|+\left|Y\right|/2+ \ldots+1 \leq 2\left|Y\right|$ variables and $\leq \log(\left|Y\right|)$ steps.
\end{proof}

\begin{proposition}\label{longconprop}
\label{le2}
The functions $\wedge:2^X \rightarrow \left\{0,1\right\}$ and $\vee:2^X \rightarrow \left\{0,1\right\}$, defined by
\[ \wedge(x)=\wedge(x_1,\ldots, x_n)=x_1 \wedge \ldots \wedge x_n, \]
\[ \vee(x)=\vee(x_1, \ldots,x_n)=x_1 \vee \ldots \vee x_n, \]
for $x=(x_1, \ldots,x_n) \in 2^X$, can be calculated by bi-quadratic cooperative Boolean circuits with $\leq 2\left|X\right|$ variables in $\leq \lceil\log(\left|X\right|)\rceil$ steps.
\end{proposition}

\begin{proof}
Group variables into pairs and use a  cooperative bi-quadratic Boolean circuit of depth~1 to calculate $x_1 \wedge x_2, x_3 \wedge x_4, \dots , x_{n-1} \wedge x_n$.  If~$n$ is odd, use $x_n \wedge x_n$ for the last pair instead.
Then use induction and concatenation of Boolean circuits.  The total number of internal variables can be estimated as in the proof of Proposition~\ref{le1}.
\end{proof}

\begin{corollary}\label{techlem}
For every positive integer $r$ there exist positive integers $u(r)$ and $d(r)$ such that every cooperative Boolean function
 $f: 2^r \rightarrow 2^r$ can be calculated by a cooperative bi-quadratic Boolean circuit with at most $u(r)$ internal variables that calculates~$f$ in $d(r)$ steps.
\end{corollary}

\begin{proof}
Note that the Conjunctive or the Disjunctive Normal Form of cooperative Boolean functions do not use negations and use Propositions~\ref{le1}
and ~\ref{le2}.
\end{proof}

\subsection{The proof of Lemma~\ref{f1lem0}}\label{Lemmasf1lem0sec}

Assume the pair $(k, \eps)$ is $c$-friendly and $(1 + \eps)\log(n) = k\ell$.  Let $F_1^n$ be a function that takes as input a pair of vectors $(s_{X}, s_{R})$ with $|X| = \log(n)$ and $|R| \leq \gamma_R \log(\log(n))$ for a fixed constant $\gamma_R$ such that $F_1^n$ returns the code for $V(s_X) + 1$ modulo $n - i$ whenever $s_X$ is coding and $s_R$ is the code for~$i$.   We need to construct a cooperative bi-quadratic Boolean circuit $B = B(X \cup R, H, Z)$ of depth $\leq \gamma \log(n)$ with $|H| \leq \gamma (\log(n))^2$ for some fixed $\gamma$ that calculates~$F_1^n$.

Since all codes are pairwise incomparable, by Proposition~\ref{extendprop} we may assume that $F_1^n$ is cooperative, and hence can be calculated by a cooperative bi-quadratic Boolean circuit.  Unfortunately, we cannot use Lemma~\ref{techlem} directly, since it does not give us sufficient control on how fast the depth and number of internal variables of this circuit grow with~$n$.  Our strategy in this proof will be to break down the codes of the integers involved into small chunks of fixed size, use Lemma~\ref{techlem} to build Boolean circuits of fixed size and depth that calculate certain auxiliary functions $g_1, g_2, g_3, g_4$ on these chunks, and to use
Propositions~\ref{duplprop} and~\ref{longconprop} to calculate auxiliary functions~$h_1, h_2$ that keep track of the global picture and allow us to concatenate the smaller Boolean circuits into a larger one with the desired properties.

 Let $D \subset C_k$ be such that $|D| = K = 2^{\frac{k}{1 + \eps}}$ and let
$w: D \rightarrow \{0, \ldots , K-1\}$ be a bijection. In this proof it will be convenient to use the notation $x$ for the input vector $s_X$ and the notation $r$ for the input vector $s_R$ of our Boolean circuit.  Then $x \in 2^{k\ell}$ and wlog we may assume that $r \in 2^{km}$ for some $m < \ell$. A vector $x$ ($r$) is coding if
$x \in  D^\ell$ ($r \in  D^m$). We partition
$x, r$ into consecutive vectors~$x^j \in 2^{X^j}, r^j \in 2^{R^j}$ of length~$k$ each (as in Section~\ref{codingsec}). The corresponding decoding function $V$ for~$x$ is then defined by

\begin{equation}\label{Vdef}
V(x) = \sum_{j=0}^{\ell-1} w(x_j)K^j.
\end{equation}

Consider four functions with the following properties:

\begin{itemize}
\item $g_1: C_k \rightarrow 2^2$ returns~$(1,0)$ for the input $x^j \in D$ with $w(x^j) = K-1$ and returns~$(0, 1)$ on all other inputs in~$D$.
\item $g_2: C_k \times 2^2 \rightarrow C_k$ is such that $g_2(x^j, 1, 0), g_2(x^j, 0, 1) \in D$ with $w(g_2(x^j, 1, 0)) = w(x^j) + 1$ and $w(g_2(x^j, 0, 1)) = x^j$ whenever~$x^j \in D$.
\item $g_3: C_k \times C_k \rightarrow 2^2$ is such that $g_3(x^j, x^j) = (1, 0)$ whenever $x^j \in D$, and $g_3(x^j, y^j) = (0, 1)$ whenever $x^j, y^j \in D$ with $x^j \neq y^j$.
\item $g_4: C_k \times 2^2 \rightarrow C_k$ is such that $g_4(x^j, 1, 0), g_4(x^j, 0, 1) \in D$ with $w(g_4(x^j, 1, 0)) = 0$ and $w(g_4(x^j, 0, 1)) = w(x^j)$ whenever $x^j \in D$.
\end{itemize}

Since the requirements we have specified for these functions only pertain to subsets~$D, D \times \{(0,1),(1,0)\}$, or $D \times D$  of the domains that consist of pairwise incomparable vectors, we may assume that these functions are cooperative.  Thus they can be calculated by cooperative, bi-quadratic Boolean circuits $B_1, B_2, B_3, B_4$ whose sizes and depths are constant in the sense that they do not depend on~$n$.

Let $h_1: (2^2)^\ell \rightarrow (2^2)^\ell$ be a cooperative function such that the $j$-th coordinate of its output is $(1,0)$ whenever the $i$-th coordinates of its input vector are $(1,0)$ for all~$i < j$, and is $(1,0)$ whenever at least one of the $i$-th coordinates of its input vector is $(0, 1)$ for some~$i < j$.  By Proposition~\ref{longconprop}, such~$h_1$ can be calculated by a cooperative, bi-quadratic Boolean circuit $B^1$ with $\leq 4\left|\ell\right|$ variables in $\leq \lceil\log(\left|\ell\right|)\rceil$ steps.  If we compose a product of~$\ell$ copies of the functions~$g_1$ on individual components~$x^j$ of~$x$ with~$h_1$, we obtain a vector in~$(2^2)^\ell$ that indicates the location of carry-over digits in the operation~$w(x) + 1$ as performed on the codes, with $(1,0)$ indicating that $w(x^j)+1$ needs to be calculated, whereas~$(0, 1)$ signals that~$x^j$ needs to be copied.  This composition can be calculated by a bi-quadratic cooperative Boolean circuit~$B_c$ with $\leq \gamma_1 \left|\ell\right|$ variables and depth $\leq \gamma_1\lceil\log(\left|\ell\right|)\rceil$, where the constant~$\gamma_1$ depends on the size and depth of~$B_1$.

If we first duplicate the value of~$x$, keep copying it until~$B_c$ finishes its calculations, and then use the copies of~$x^j$ as the first inputs in copies of~$B_2$, with the coordinates of the output of~$B_c$ providing the second inputs, we can construct a cooperative bi-quadratic Boolean circuit~$B_a$ with $\leq \gamma_2 \left|\ell \right|$ variables and depth $\leq \gamma_2 \lceil\log(\left|\ell\right|)\rceil$ that calculates the code for~$w(x) + 1 \ mod \ n$ for every coding vector~$x$.  Here $\gamma_2$ is another constant that depends on~$\gamma_1$ and the size and depth of~$B_2$.

It remains to compare the output of~$B_a$ with $n - i$, where $i$ is coded by $r$ and reset it to zero if in fact the two vectors are equal.  So far, we have not specified how $i$ is coded by the elements of the counter; it will be most convenient if we assume that~$r$ codes the last $mk$ binary digits of~$n-i$ according
to~(\ref{Vdef}).  Since $m \ll \ell$, we have a code for~$i$ in this sense iff for the partition of the output of~$B_a$ into consecutive~$y^j$s we have $y^j = r^j$ for all $j < m$ and $w(y^j) = K-1$ for $j \geq m$.  We can use a product of~$m$ functions of the form~$g_3$ with~$\ell - m$ functions of the form~$g_1$ to code the equality as a vector in~$(2^2)^\ell$ all of whose coordinates are~$(1,0)$, while inequality will be signified by at least one coordinate of the form~$(0, 1)$.  Composing this product with a cooperative function~$h_2: (2^2)^\ell \rightarrow 2^2$ that returns~$(1,0)$ on the code for equality and~$(0, 1)$ on all codes for inequality results in a function that detects a code for~$n-i$.  Again, by Proposition~\ref{longconprop}, this composition can be calculated by a cooperative, bi-quadratic Boolean circuit~$B_e$ with $\leq \gamma_3 \left|\ell\right|$ variables and depth $\leq \gamma_3\lceil\log(\left|\ell\right|)\rceil$, where the constant~$\gamma_3$ depends on the size and depth of~$B_3$.

To put it all together, we need to copy the value of~$r$ until it will be needed for the calculation of~$B_e$, and also retain a copy~$y$ of the output of~$B_a$ until~$B_e$ has finished its calculations.  Finally, we can use the coordinates of that copy of~$y$ and of the output of~$B_e$ as inputs to $\ell$ copies of the Boolean
circuit~$B_4$ that calculates~$g_4$.  This adds $\leq \gamma_4 \left|\ell\right|$ variables and  $\leq \gamma_4\lceil\log(\left|\ell\right|)\rceil$ steps to the final Boolean circuit~$B$.  Since~$\log(n) \leq \ell$, we conclude that~$B$  will have depth $\leq \gamma \log(n)$ and $\leq \gamma (\log(n))^2$ variables for some fixed $\gamma$.

It is straightforward to verify that~$B$ is cooperative, bi-quadratic, and calculates~$F_1^n$ as required. $\Box$

\subsection{The proof of Lemma~\ref{countlem0}}\label{countlemsec}

Lemma~\ref{countlem0} follows from Lemma~\ref{countlem+} below by letting $R$ be the set of  variables for our counter and $T = |I|$, with $g$ coding the sequence of desired values of the counter.  By~(\ref{Rsize}) and the choice of~$I$ we have  $|R|T \leq  (log(n))^\alpha$ for all $\alpha > 1$ and sufficiently large~$n$, and therefore $B_c$ satisfies the  restrictions on its depth and number of variables that are implied by Lemma~\ref{countlem0}.

\begin{lemma}\label{countlem+}
Let $0 < q < 1$.  Then there exists a positive constant $\gamma$ such that for every nonempty set $R$ and positive integer $T$ and every function
$g: \{0, \dots , T-1\} \rightarrow 2^R$ there exists a Boolean input-output system $B_c = B_c(\emptyset, H, R)$ with the following properties:

\smallskip

\noindent
(i) $|H| \leq \gamma |R| T\log(|R|T)$;

\smallskip

\noindent
(ii) $d := d(B_c) \leq \gamma \log(|R|T)$;

\smallskip

\noindent
(iii) $B_c$ is cooperative and bi-quadratic;

\smallskip

\noindent
(iv) There exists a fixed state $s^*_H \in 2^H$ such that for a proportion of $> q$ of all possible initial states of the variables in~$H$ the following two conditions hold:

\smallskip

\noindent
(v) $s_H(2d+1) = s^*_H$;

\smallskip

\noindent
(vi) For all times $t \geq d$ the system writes the value $g(t\ mod \ T)$ to its output variables~$R$.
\end{lemma}

\bigskip

\begin{proof}
 Let $q, R, T, g$ be as in the assumption. Note that we may wlog assume that $|R| T$ is sufficiently
large, since we can replace $T$ with $T^+ = kT$, choose $g(t) = g(t - T)$ for $t \geq T$, and make $\gamma$ sufficiently large to cover the  finite
set of instances whose  original size $|R| T$ was too small for our probability estimates.

In order to make the idea of our construction more transparent, let us begin by specifying the dynamics on the highest level $H(d-1)$ and the output variables.  The set $H(d-1)$ will be partitioned into pairwise disjoint subsets $H_j$ for $j = 0, \ldots , T-1$, each of size $|R|$. For any $t > d-1$, let $t^* = t - d + 1 \ mod \ T$.

Let us enumerate $H_j = \{h_{i,j}: \  i \in [|R|]\}$ for $j \in \{0, \ldots , T-1\}$ and let us provisionally define

\begin{equation}\label{provisc}
h_{i,j}(t+1) = h_{i,j-1}(t),
\end{equation}

\bigskip

where we interpret $j - 1 = -1$ as $T-1$.  This definition will give us the correct output sequence in the sense of~(vi)
 as long as $s_{H_j} (d-1) = g(t)$ for all $t \in \{0, \ldots , T-1\}$.  In other words, we need to make sure that, with sufficiently high probability,   each of the variables $s_{h_{i,j}}$ assumes a certain value $h^*_{i,j}$ at time $d-1$, and that at all subsequent times the dynamics
of $H(d-1)$ is identical to the dynamics specified by~(\ref{provisc}). If we can achieve this, then we can arrange for the output vector $s_R(t+1)$ to simply be a copy of $H_{T-1}(t)$.

The following trick allows us to achieve this goal. Let

\[ P_0=\{ h_{i,j} \in H(d-1): h^*_{i,j}=0 \}, \ P_1=\{ h_{i,j} \in H(d-1): h^*_{i,j}=1 \}. \]

We will construct~$B_c$ in such a way that there exist variables $k_{i,j} \in H(d-2)$, where $H(d-2)$ is the level right below $H(d-1)$, $i \in |R|$, $j<T$ such that with probability $>q$:

\begin{equation}\label{Hd2dynamics}
\begin{split}
\forall h_{i,j} \in P_0 \ s_{k_{i,j}}(d-2) = 0 \ &\& \ \forall t > d-2 \ s_{k_{i,j}}(t) = 1, \\
\forall h_{i,j} \in P_1 \ s_{k_{i,j}}(d-2) = 1 \ &\& \ \forall t > d-2 \ s_{k_{i,j}}(t) = 0.
\end{split}
\end{equation}

Now we modify~(\ref{provisc}) as follows:

\begin{equation}\label{modifc}
\begin{split}
\mbox{If} \ h_{i,j} \in P_0, \ \mbox{then} \ s_{h_{i,j}}(t+1) = s_{h_{i-1,j}}(t) \vee s_{k_{i, j}} (t),\\
\mbox{If} \ h_{i,j} \in P_1, \ \mbox{then} \ s_{h_{i,j}}(t+1) = s_{h_{i-1,j}}(t) \wedge  s_{k_{i, j}} (t).
\end{split}
\end{equation}

Notice that in view of~(\ref{Hd2dynamics}), our revised definition~(\ref{modifc}) of~(\ref{provisc}) ensures that with probability~$>q$, at time~$d-1$ all variables $h_{i,j}$ will take their desired values, and the input from level~$H(d-2)$ will not influence the  dynamics at level $H(d-1)$ at any time $t> d-1$.

Therefore, given $q<1$, a set $P$ of sufficiently large size, and a partition $P_0 \cup P_1$ of~$P$, it suffices to construct a system $B_c(\emptyset,H,P)$ of depth~$d$, and $s^*_H  \in 2^H$ such that for some fixed $\gamma>0$ that does not depend on~$|P|$, with probability~$> q$:

\begin{enumerate}[(a)]
\item $\left| H \right| \leq \gamma |P| \log(|P|)$;
\item $d \leq \gamma \log(|P|)$;
\item for $p_0 \in P_0$ we have $s_{p_0}(d)=0$ and $s_{p_0}(t)=1$ for all $t\geq d+1$;
\item for $p_1 \in P_1$ we have $s_{p_1}(d)=1$ and $s_{p_1}(t)=0$ for all $t\geq d+1$;
\item $s_H(2d+1)=s^*_H$,
\end{enumerate}

where points (c), (d), (e) will hold  with probability~$> q$.

We show how to construct such $B_c(\emptyset,H,P)$, provided that $P_0$ is empty. If this is not the case, we can construct a dual system to take care of the case $P_1 = \emptyset$ in an analogous way and take the disjoint union of the two systems.

Let $H_l$ be the lowest level of $H$, let $H_h$ be the highest level, and let $H_r$ denote all the remaining levels. Wlog we assume that $|P|/16$ is an integer, and that $H_l$ and~$H_h$ consist of $K=4\left|P\right|$ variables. Let

\[ P=\{p(i): i \leq  K/4 \}, \quad H_l=\{h_l(i): i \leq K \}, \quad H_h=\{ h_h(i): i \leq K \}. \]

Now let
\begin{equation}\label{H*numb}
\{h^*(1), \ldots , h^*(K)\}
\end{equation}
be an enumeration of the variables $h_h(1),h_h(2) \ldots,  h_h(K/2)$ such that each variable in this set gets listed exactly twice. Similarly, let
\begin{equation}\label{H**numb}
\{h^{**}(1), \ldots , h^{**}(K/4)\}
\end{equation}

be an enumeration of the variables $h_h(9K/16 + 1), \ldots, h_h(11K/16)$, such that each variable in this set gets listed exactly twice.
The regulatory functions for the lowest level $H_l$ are given by

\begin{equation}\label{HZ(0)reg}
s_{h_l(i)}(t+1) = s_{h_l(i)}(t) \wedge s_{h^*(i)}(t), \, i \leq K
\end{equation}

and for the output variables by

\begin{equation}\label{HZ(0)out}
s_{p_i}(t+1) = s_{p_i}(t) \wedge s_{h^{**}(i)}(t), \, i \leq K/4.
\end{equation}

Now let $H_r$ be an implementation of the sorting function $f_s:H_l \rightarrow H_h$, that is, a function such that the number of zeros in $x$ is the same as in $f(x)$ but all the zeros in $f(x)$ precede all the ones in $f(x)$. By a result in~\cite{AKS}, there exists a \emph{sorting network} of depth $O(\log(K))$ that will implement $f_s$. A sorting network performs at each step a certain number of pairwise comparisons on disjoint sets of two variables and switches the variables if they
are in the wrong order.  In the Boolean context, this operation can be implemented as $sw(s_i, s_j) = [s_i \wedge s_j, s_i \vee s_j]$; a variable that does not participate in any comparison at a given step will simply be copied.  Thus the function~$f_s$ can be calculated  by a cooperative bi-quadratic Boolean circuit of depth $O(\log(K))$ with $O(K\log(K))$ variables, so the system $B_c(\emptyset, H,P)$ satisfies (a) and (b).

Let us make a few observations about the construction that we have described.  First, note that we will end up with a Boolean input-output system rather than a Boolean circuit since there is feedback from levels~$H_h$ and~$H_l$ itself to level~$H_l$.
This will be the only feedback loops, the variables at all other levels of~$H$ will take input only from the next lower level.

Now define for all times $t \geq 0$:

\begin{equation}\label{Sdef}
S(t) =K- \sum_{i \leq K} s_{h_l(i)}(t).
\end{equation}

Note that the self-feedback~(\ref{HZ(0)reg}) at level~$H_l$ implies that

\begin{equation}\label{Smon}
\forall t \ S(t+1) \geq S(t).
\end{equation}

For a randomly chosen initial state, $S(t)$ is a random variable.   Note that $S(0)$ is a binomial random variable that counts the number of failures (zeros) in $K$ independent trials with success probability~$0.5$.  In contrast, $S(1)$ is a binomial random variable that counts the number of failures in $K$ independent trials with success probability~$0.25$. A straightforward application of the Central Limit Theorem shows that

\[ Pr (S(0)<9K/16 \ \& \ S(1)>11K/16)> {q} \]

for all sufficiently large $K$.

This means that for a proportion of $>q$ states we will have $s_{h_h(i)}(d-1)=1$ for $9/16 < i \leq 11/16$, and $s_{h_h(i)}(d)=0$ for $9/16 < i \leq 11/16$, or

\[ s_{p(i)}(d)= 1 \quad \mbox{for all} \quad i \leq 1/4K, \]
\[ s_{p(i)}(d+1)=0 \quad \mbox{for all} \quad  i \leq 1/4K. \]

By (\ref{Smon}), we get that

\[ s_{p(i)}(t)=0 {\rm \ for \ } i \leq K/4 {\rm \ and \ } t\geq d+1, \]

so (c) and (d) hold. Also, $s_{h_h(i)}(d)=0$ for $i \leq K/2$, so $s_{h_l(i)}(d+1)=0$ for $i \leq K$. By (\ref{Smon}) again,
\[ s_{h_l(i)}(t)=0 {\rm \ for \ all \ } i \leq K {\rm \ and \ } t \geq d+1,\]
which clearly implies (e).
\end{proof}

\subsection{The proof of Lemma~\ref{f2lem0}}\label{f3lemsec}

Recall the definition of crude vectors from the end of Section~\ref{codingsec} and let $r = |R_i|$. The following technical result is identical with Lemma~\ref{f2lem0}, except that we relabeled the input vectors for convenience as $z \in 2^{k\ell + r}$ instead of $(s_{X_{i_0}}, s_R)$, with $z\upharpoonright 2^{k\ell}$ playing the role 
of $s_{X_{i_0}}$, and consider the output vector simply as an element of~$2^{k\ell}$ instead of specifying $X_{i_1}$ as its set of coordinates.

\begin{lemma}\label{f3lem}
Fix $k$. For every positive integer $\ell$ let $\cZ$ be the set of all vectors $z \in 2^{k\ell + r}$ that are crude or such that
$x = z\upharpoonright 2^{k\ell}$ is coding. Fix a coding vector $x^* \in 2^{k\ell}$.  Then there exists an $L$-function
$F_3: 2^{k\ell + r} \rightarrow 2^{k\ell}$  that satisfies for all inputs~$z \in \cZ$:

\begin{itemize}
\item $F_3(z)=x = z\upharpoonright 2^{k\ell}$ if $z$ is coding;
\item $F_3(z)=x^*$ if $z$ is crude.
\end{itemize}
\end{lemma}

\begin{proof}
Let $a = \ell + \frac{r}{k} - 1$.
For $z=(z^0, \ldots, z^a)$ define
\[ ONE(x)=\vee(\wedge(z^0), \ldots, \wedge(z^a)), \qquad
 ZERO(x)=\wedge(\vee(z^0), \ldots, \vee(z^a)).\]
Observe that $ONE(z)=1$ implies that $z$ is crude or $z \notin \cZ$ and $ZERO(z)=0$ also implies that $z$ is crude or $z \notin \cZ$.

For simplicity of notation, identify~$x^*$ with the set of coordinates where this vector takes the value~1.
Define the output of~$F_3$ coordinatewise for $i < k\ell$ by

\begin{equation}\label{x*trickeqn}
\begin{split}
(F_3(z))_i &=z_i \vee ONE(x) \qquad \mbox{for} \qquad i\in x^*\\
(F_3(x))_i &=z_i \wedge ZERO(x) \qquad \mbox{for} \qquad i\notin x^*.
\end{split}
\end{equation}

Now, if $z \in \cZ$ is not crude, then $x = z\upharpoonright 2^{k\ell}$ is coding and so $x_i \vee ONE(x)=x_i$, $x_i \wedge ZERO(x)=x_i$, and $F_3(z)=x$.  On the other hand, because of the choice of clauses in~(\ref{x*trickeqn}), $F_3(z)=x^*$ if $z$ is crude.

Building a cooperative Boolean circuit of suitable depth and size that depend on~$\ell$ and hence on~$n$ with $(1+\eps)\log(n) = k\ell$ that implements the calculations of $ONE(z)$, $ZERO(z)$, and~(\ref{x*trickeqn}) is straightforward using the technique of the proof of Lemma~\ref{f1lem0}.
\end{proof}

\section{Conclusion and future directions}\label{concludesec}

We have constructed examples of cooperative Boolean systems with synchronous updating in which most initial conditions belong to the basin of attraction of a single  attractor that can be of length $> c^N$ for any $c < 2$ and sufficiently large~$N$.  This contrasts sharply with results on
various other types of dynamical systems that show nongenericity or absence of non-steady state attractors under the assumption of cooperativity.  These systems are also bi-quadratic.  It is shown in~\cite{JE} that upper bounds on the proportion of monic regulatory functions imply nontrivial upper bounds on~$c$ for bi-quadratic cooperative Boolean systems with an attractor of length~$> c^N$; in particular, if this proportion is zero, then $c < 10^{1/4}$. In the follow-up paper~\cite{JMII} we will re-examine this upper bound in the context of $p$-$c$-chaos.

Exponentially long attractors in Boolean systems are a hallmark of chaotic dynamics.  Two other important hallmarks of chaos are a small fraction of eventually frozen nodes and high sensitivity to initial conditions.
Sensitivity to initial conditions can be formalized in a variety of ways.
For example, let us call a Boolean network~\emph{$p$-unstable} if a random single-bit flip in a randomly chosen initial condition moves the system to the basin of attraction of a different attractor with probability~$>p$. The systems that we constructed in this paper are extremely chaotic with respect to attractor length and the expected proportion of eventually frozen nodes, but also $p$-coalescent, which contradicts $1-p$-instability. In the follow-up paper~\cite{JMII} we will
construct, for any given $0 < p < 1 < c < 2$,  examples of Boolean systems that are simultaneously $p$-$c$-chaotic and $p$-unstable.
However, Theorem~5 of~\cite{JE} shows that for $\sqrt{3} < c$  and $p > 0.75 + \frac{\ln 0.5c}{2\ln 0.75}$ no cooperative Boolean system that uses only binary AND and binary OR as regulatory functions can be simultaneously $p$-unstable and have an attractor of length~$\geq 0.75$. We will improve upon this result in the context of~$p$-$c$-chaos and also explore the relation to other notions of sensitive dependence on initial conditions in~\cite{JMII}.

\end{document}